\documentclass{amsart}
\usepackage{amssymb,amstext,amsmath,amscd,amsthm,amsfonts,enumerate,graphicx,latexsym,stmaryrd,multicol}
\usepackage{hyperref}
\usepackage{comment} 
\usepackage{cleveref}  

\usepackage{tikz-cd} 
\usepackage[all]{xy}
\usepackage[notcite,notref]{} 
\usepackage[notcite, notref]{}
\usepackage[notcite, notref]{}
\usepackage[notcite,notref]{} 

\newtheorem{thm}{Theorem}[section]
\newtheorem{lem}[thm]{Lemma}
\newtheorem{prop}[thm]{Proposition}

\theoremstyle{definition}
\newtheorem{dfn}[thm]{Definition}

\newtheorem{rem}[thm]{Remark}
\newtheorem{conv}[thm]{Convention}

\newtheorem{chunk}[thm]{\hspace*{-1.065ex}\bf}

\theoremstyle{remark}

\numberwithin{equation}{thm}

\def\ann{\operatorname{ann}}
 
\def\p{\mathfrak{p}}

\def\depth{\operatorname{depth}} 
\def\Ext{\operatorname{Ext}}

\def\Hom{\operatorname{Hom}}
\def\m{\mathfrak{m}}
\def\syz{\mathrm{\Omega}}

\def\cm{\operatorname{CM}}

\def\depth{\operatorname{depth}}

\def\e{\operatorname{\mathbb{E}}}

\def\Ext{\operatorname{Ext}}
\def \ca {\operatorname{ca}}

\def\Hom{\operatorname{Hom}}

\def \c {\mathcal C}

\def\m{\mathfrak{m}}

\def\mod{\operatorname{mod}}
\def\Mod{\operatorname{Mod}}

\def\p{\mathfrak{p}}

\def\supp{\operatorname{Supp}}
\def\syz{\mathrm{\Omega}}
   
\def\t{\operatorname{\mathbb{T}}}
\def\m{\mathfrak{m}}
\def\syz{\mathrm{\Omega}}
 
\def\Tor{\operatorname{Tor}}
\def\tr{\operatorname{Tr}}

\def\syz{\Omega}
  
\def\tr{\operatorname{tr}} 
\def \Tr {\operatorname{Tr}}

\def\X{\mathcal{X}}
\def\Y{\mathcal{Y}}

\def\e{\operatorname{\mathbb{E}}} 
 
\def\t{\operatorname{\mathbb{T}}}

\begin{document}
\title{Trace ideal and annihilator of Ext and Tor of regular fractional ideals, and some applications}      

\author{Souvik Dey}
\address{Souvik Dey\\ Department of Mathematics \\ University of Kansas\\405 Snow Hall, 1460 Jayhawk Blvd.\\ Lawrence, KS 66045, U.S.A.}
\email{souvik@ku.edu}

\begin{abstract} Given a commutative Noetherian ring $R$ with total ring of fractions $Q(R)$, and a finitely generated $R$-submodule $M$ of $Q(R)$, we prove an equality between trace ideal, and certain annihilator of Ext and Tor of $M$. As a consequence, we answer in one-dimensional local analytically unramified case, a question raised by the present author and R. Takahashi. As another application, we give an alternative proof of a recent result of Ö. Esentepe that for one-dimensional analytically unramified Gorenstein local rings, the cohomology annihilator of Iyengar and Takahashi coincides with the conductor ideal. 
\end{abstract} 
\maketitle

\section{Introduction}

Let $R$ be a commutative Noetherian ring with unity, with total ring of fractions $Q(R)$. Let $\Mod R$ and $\mod R$ denote the category of all $R$-modules and all finitely generated $R$-modules respectively. 

Annihilators of Ext and Tor modules of certain subcategories, and their connections to generation of module categories and the singular locus of the ring, have recently caught a lot of attention. We draw the reader's attention to \cite{iy}, \cite{dim}, \cite{bl} among many instances in the literature.  

In this article, we prove a connection between annihilators of Ext and Tor and trace ideal (see Definition \ref{trd}) of a finitely generated $R$-module $M$ where $M$ is an $R$-submodule of $Q(R)$ containing a non-zero-divisor of $R$. Namely, we prove the following main result (see Theorem \ref{traceann}), where $\tr_R(M)$ denotes trace ideal. 

\begin{thm}\label{mai} Let $M$ be a finitely generated $R$-submodule of $Q(R)$ containing a non-zero-divisor of $R$. Let $\syz_R M$ be the first syzygy in some projective resolution of $M$. Then, it holds that

\begin{align*} \tr_R(M)=\bigcap_{i>0, N\in \mod R} \ann_R \Tor^R_i(M,N)=\bigcap_{i>0, N\in \Mod R} \ann_R \Ext^i_R(M,N)=\ann_R\Ext^1_R(M,\syz_R M) 
\end{align*}  

\end{thm}

Using this main Theorem, we give the following affirmative answer to \cite[Question 4.7]{dt} when $R$ is a one-dimensional local ring with reduced completion (see Proposition \ref{41}).  

\begin{thm} Let $R$ be a one-dimensional local ring with reduced completion. Let $\c(R)$ denote the conductor of $R$ in $\overline R$. Then, $$\c(R)=\bigcap_{i>0, M,N\in \cm_0(R)} \ann_R \Tor^R_i(M,N)=\bigcap_{i>0, M,N\in \cm_0(R)} \ann_R \Ext^i_R(M,N)=\bigcap_{i>1, M,N\in \cm_0(R)}  \ann_R  \Ext^i_R(M,N)$$ 
\end{thm}    

Here, $\cm_0(R)$ is the subcategory of all maximal Cohen--Macaulay modules whose localiztions at all non-maximal prime ideals are free.  
We note here that \cite[Question 4.7]{dt} was motivated by the question of whether we have $(c)\implies (b)$ in \cite[Theorem 1.1(1)]{dim} (also see \cite[Theorem 1.1]{dt} in this regard).  

Using the above two results, we also give a quick alternative proof of one of the main results of \cite{curve}, namely \cite[Theorem 4.3 and 5.10]{curve}, that for  one-dimensional local Gorenstein rings with reduced completion, the cohomology annihilator $\bigcup_{n\ge 1} \bigcap_{i\ge n, M,N\in \mod R} \Ext^i_R(M,N)$ coincides with the conductor ideal, see Proposition \ref{cacon}.   

The organization of the paper is as follows: In Section 2, we state definitions and properties of fundamental notions used in this paper. In Section 3, we prove our main result, namely Theorem \ref{mai}. In Section 4, we give applications of our main result. 

\section{Preliminaries} 

In this section, we lay out the basic conventions, definitions, and preliminary discussions which will be used throughout the rest of the paper. 

\begin{conv} Throughout, $R$ will denote a commutative Noetherian ring with unity, with total ring of fractions $Q(R)$. Let $\Mod R$ denote the category of all $R$-modules. All subcategories of $\Mod R$ are assumed to be strict (closed under isomorphism) and full and are assumed to contain the zero module. So in particular, our subcategories are only determined by the collection of objects (modules) only. By $\mod R$, we will denote the subcategory of $\Mod R$ consisting of all finitely generated modules. We sometimes may abbreviate "finitely generated module" as "finite module" only. When $M$ is a finitely generated $R$-module, we denote by $\syz^n_R M$ the $n$-th syzygy in some projective resolution of $M$. When $R$ is moreover local, we take $\syz^n_R M$ to be the $n$-th syzygy in the minimal free resolution of $M$ by finitely generated free $R$-modules.

We also denote by by $\cm(R)$ the subcategory of $\mod R$ consisting of maximal Cohen--Macaulay modules (recall that an $R$-module $M$ is called {\em maximal Cohen--Macaulay} if $\depth_{R_\p}M_\p=\dim R_\p$ for all $\p\in\supp_RM$). We say that $R$ is Cohen--Macaulay if $R\in \cm(R)$. When $(R,\m)$ is local, by $\cm_0(R)$, we denote the subcategory of all modules in $\cm(R)$ whose localizations at all non-maximal prime ideals are free. 
\end{conv}   

\begin{dfn}
Let $\X,\Y$ be subcategories of $\Mod R$, and let $n\ge0$ be an integer.  

Adopting the notation of \cite[Definition 3.1]{dt}, we define the ideals $\t_n(\X,\Y)$ and $\e^n(\X,\Y)$ of $R$ by
\begin{align*}
\t_n(\X,\Y)&=\bigcap_{i>n}\bigcap_{X\in\X}\bigcap_{Y\in\Y}\ann_R\Tor_i^R(X,Y),\\
\e^n(\X,\Y)&=\bigcap_{i>n}\bigcap_{X\in\X}\bigcap_{Y\in\Y}\ann_R\Ext_R^i(X,Y).
\end{align*}
We put $\t_n(\X)=\t_n(\X,\X)$ and $\e^n(\X)=\e^n(\X,\X)$.
\end{dfn} 

\begin{dfn} Following \cite[Definition 2.1]{iy}, we define for any integer $n\ge 1$, the ideal   $$\ca^n(R):=\bigcap_{X,Y\in \mod R} \bigcap_{ i\ge n} \ann_R \Ext^i_R(X,Y) $$ we also put $\ca(R):=\bigcup_{n\ge 1} \ca^n(R)$. It is clear that $\ca^{n}(R)=\e^{n-1}(\mod R,\mod R)$. Moreover, it is also clear that $\ca^n(R)\subseteq \ca^{n+1}(R)$ for all $n\ge 1$, hence this chain of ideals stabilizes since $R$ is Noetherian, so we get $\ca(R)=\ca^s(R)$ for all big enough $s>0$. 
\end{dfn}

\begin{chunk}\label{conductor} Let $Q(R)$ be the total ring of fractions of $R$, and $\overline R$ be the integral closure of $R$ in $Q(R)$. We call finitely generated $R$-submodules of $Q(R)$ to be fractional ideals, and those fractional ideals which also contain a non-zero-divisor of $R$ are called regular fractional ideals.   

For $R$-submodules (not necessarily finitely generated) $M,N$ of $Q(R)$, we denote by $(N:M)$ the $R$-submodule of $Q(R)$ defined as $(N:M) :=\{x\in Q(R): xM\subseteq N\}$.
If $M,N$ are $R$-submodules of $Q(R)$ and $M$ contains a non-zero-divisor of $R$, then $(N:M)\cong \Hom_R(M,N)$, see \cite[Proposition 2.4(1)]{trace}. We will use this identification freely throughout the article without possible further reference.    
We also put $\c(R):=(R:\overline R)$, and we call it the conductor of $R$. We note that $\c(R)$ is an ideal of both $R$ and $\overline R$: indeed, since $1\in \overline R$, so $\c(R)\subseteq R$, hence $\c(R)$ is an $R$-submodule of $R$, so an ideal of $R$; and similarly $\c(R)\overline R\cdot\overline R=\c(R)\overline R\subseteq R$, so $\c(R)\overline R\subseteq (R:\overline R)=\c(R)$, hence $\c(R)$ is an ideal of $\overline R$. Moreover, if $I$ is an ideal of both $R$ and $\overline R$, then $I\subseteq \c(R)$: indeed, if $I$ is an ideal of $R$ and $\overline R$, then $I\overline R\subseteq I\subseteq R$, so $I\subseteq (R:\overline R)=\c(R)$. It is clear that $\c(R)$ contains a non-zero-divisor if and only if $\overline R$ is module finite over $R$ (for example, when $R$ is a local ring whose completion is reduced, see \cite[Theorem 4.6(i)]{lw}).  
\end{chunk}    

\begin{dfn}\label{trd} For an $R$-module $M$, the trace ideal of $M$ is $\tr_R(M)=\sum_{f\in \Hom_R(M,R)}\text{Im}(f)$, which is the ideal of $R$ generated by all homomorphic images of $M$ into $R$.
\end{dfn}   

\begin{chunk}\label{trcon} If $M$ is an $R$-submodule (not necessarily finitely generated) of $Q(R)$ containing a non-zero-divisor of $R$, then it follows that $\tr_R(M)=(R:M)M$, see \cite[Proposition 2.4(2)]{trace}. So in particular, $\tr_R(\overline R)=(R:\overline R)\overline R=(R:\overline R)=\c(R)$. Hence, $\tr_R(\c(R))=\c(R)$ (see \cite[Proposition 2.8(iv)]{lindo}, for this part of the result of \cite[Proposition 2.8]{lindo}, one does not need $M$ to be finitely presented). We will use this fact $\tr_R(\c(R))=\c(R)$ throughout the rest of the paper without possible further reference. 
\end{chunk}   

\begin{dfn} Given a subcategory $\X$ of $\mod R$, and integer $n\ge 1$, by $\syz^n \X$ we denote the collection of all modules for which there exists an exact sequence of the form $0\to M \to P_{n-1} \to \cdots \to P_0 \to N\to 0$ for some $N\in \X$ and some finitely generated projective $R$-module $P_0,...,P_{n-1}$.    
\end{dfn}

\begin{chunk} Note that if $M\in \mod R$, then $M^*\in \syz^2 \mod R$. If $R$ is Cohen--Macaulay, then it is clear that $\syz^n \cm(R) \subseteq \cm(R)$ for all $n\ge 1$. If $R$ is Cohen--Macaulay of dimension $d$, then $\syz^n \mod R\subseteq \cm(R)$ for all $n\ge d$.      
\end{chunk}

\begin{dfn}  For a finitely generated $R$-module $M$ we denote by $\Tr M$ the {\em (Auslander) transpose} of $M$.
This is defined as follows.
Take a projective presentation $P_1\xrightarrow{f}P_0\to M\to 0$ by finitely generated projective modules $P_1,P_0$.
Dualizing this by $R$, we get an exact sequence $0\to M^\ast\to P_0^\ast\xrightarrow{f^\ast}P_1^*\to\Tr M\to0$, that is, $\Tr M$ is the cokernel of the map $f^\ast$. It is clear that $\Tr M$ is also finitely generated.   
The transpose of $M$ is uniquely determined up to projective summands; see \cite{AB} for basic properties.
\end{dfn}    

\section{Main result}

In this section, we prove our main result. For this, we first need the following Lemma, part of which was essentially shown in \cite[Lemma 2.14]{iy}.   

\begin{lem}\label{lemm} Let $M$ be a finitely generated module over $R$. Let $\syz_R M$ be the first syzygy in some projective resolution of $M$. Then, we have equalities $\t_0(M,\Mod R)=\ann_R \Tor^R_1(M,\Tr M)=\e^0(M,\Mod R)=\ann_R\Ext^1_R(M,\syz_R M)=\{x\in R:  M\xrightarrow{\cdot x} M \text{ factors through some finitely generated free } R\text{-module }\}$. 
\end{lem} 

\begin{proof} Let us call the sets (1), (2), (3), (4) and (5) in order. Clearly, $(1)\subseteq (2)$ and $(3)\subseteq (4)$.

We first prove $(2)\subseteq (3)$: By \cite[Lemma (3.9)]{Y}, we have isomorphism $\Tor^R_1(M,\Tr M)\cong \underline{\Hom}_R(M,M)$. So, if $x\in \ann_R\Tor^R_1(M,\Tr M)=\ann_R \underline{\Hom}_R(M,M)$, then since $\text{id}_M \in \Hom_R(M,M)$, so the map $x\cdot \text{id}_M: M \to M$ factors through a projective $R$-module i.e. there is a commutative diagram 
\begin{tikzcd}
M \arrow[rd] \arrow[rr, "\cdot x"] &              & M \\
                                   & P \arrow[ru] &  
\end{tikzcd}  Hence for every $R$-module $N$ and integer $i>0$, we get an induced commutative diagram by  linearity of the functor $\Ext^i_R(-,N)$  
\begin{tikzcd}
{\Ext^i_R(M,N)} &                              & {\Ext^i_R(M,N)} \arrow[ld] \arrow[ll, "\cdot x"'] \\
                & {\Ext^i_R(P,N)=0} \arrow[lu] &                                                  
\end{tikzcd} Hence $x\cdot \Ext^i_R(M,N)=0$ i.e. $x\in \ann_R \Ext^i_R(M,N)$. Since $i>0$ and $N$ were arbitrary, we get $x\in \e^0(M,\Mod R)$. This proves $(2)\subseteq (3)$. 

Next we prove $(4)\subseteq (1)$: By hypothesis, we have an exact sequence $\sigma: 0\to \syz_R M \to P \xrightarrow{f} M \to 0$ for some projective $R$-module $P$. Applying $\Hom_R(M,-)$ we get exact sequence $$0\to \Hom_R(M,\syz_R M)\to \Hom_R(M,P)\xrightarrow{\phi\mapsto f\circ \phi} \Hom_R(M,M)\xrightarrow{g} \Ext^i_R(M,\syz_R M)$$ , where $g(\text{id}_M)=\sigma$.  Now if $x\in \ann_R\Ext^i_R(M,\syz_R M)$, then $0=x\sigma=g(x\cdot \text{id}_M)$, so 

$x\cdot \text{id}_M\in \ker g=\text{Im}(\Hom_R(M,P)\xrightarrow{\phi\mapsto f\circ \phi} \Hom_R(M,M))$. So, there exists $\phi\in \Hom_R(M,P)$ such that we have a commutative diagram  
\begin{tikzcd}
M \arrow[rd, "\phi"] \arrow[rr, "\cdot x"] &                   & M \\
                                           & P \arrow[ru, "f"] &  
\end{tikzcd} Hence for every $R$-module $N$ and integer $i>0$, we get an induced commutative diagram by  linearity of the functor $\Tor^R_i(-,N)$ 
\begin{tikzcd}
{\Tor^R_i(M,N)} \arrow[rd, "\phi"] \arrow[rr, "\cdot x"] &                                   & {\Tor^R_i(M,N)} \\
                                                         & {\Tor^R_i(P,N)=0} \arrow[ru, "f"] &                
\end{tikzcd}  

Hence $x\cdot \Tor^R_i(M,N)=0$ i.e. $x\in \ann_R \Tor^R_i(M,N)$. Since $i>0$ and $N$ were arbitrary, we get $x\in \t_0(M,\Mod R)$. 

So we have now seen $(1)\subseteq (2)\subseteq (3)\subseteq (4)\subseteq (1)$. So, $(1)=(2)=(3)=(4)$. Now clearly $(5)\subseteq (1)$ by same argument as above. 

To see $(3)\subseteq (5)$: First notice that since $M$ is finitely generated, so there exists a finitely generated free $R$-module $F$ and an $R$-module $X$ such that we have an exact sequence $0\to X \to F \xrightarrow{h} M \to 0$. If $x\in \ann_R \Ext^1_R(M,X)$, then by the similar argument as in the proof of $(4)\subseteq (1)$, we see that $M\xrightarrow{\cdot x}M$ factors through $F$. This shows $(3)\subseteq (5)$. 

This finally shows all the five sets are equal.  
\end{proof}

\begin{rem} Since for a finitely generated $R$-module $M$,  $\t_0(M,\Mod R)\subseteq \t_0(M,\mod R)\subseteq \ann_R\Tor^R_1(M,\Tr M)$, hence all the ideals in Lemma \ref{lemm} are also equal to $\t_0(M,\mod R)$.  
\end{rem}

Now we prove the main result of our paper 

\begin{thm}\label{traceann} Let $M$ be a finitely generated module over $R$. Let $\syz_R M$ be the first syzygy in some projective resolution of $M$.  If $M$ is an $R$-submodule of $Q(R)$ and $M$ contains a non-zero-divisor of $R$, then $\tr_R(M)=\t_0(M,\Mod R)=\ann_R \Tor^R_1(M,\Tr M)=\e^0(M,\Mod R)=\ann_R\Ext^1_R(M,\syz_R M)$. 
    
\end{thm}    

\begin{proof} By Lemma \ref{lemm}, it is enough to show that $$\tr_R(M)=\{x\in R:  M\xrightarrow{\cdot x} M \text{ factors through some finitely generated free } R\text{-module }\}$$  Call the right hand side set (which is an ideal by Lemma \ref{lemm}) to be $L$.  

Since $\tr(M)=(R:M)M$ by \cite[Proposition 2.4(2)]{trace}, and $L$ is an ideal, so to show $\tr_R(M)\subseteq L$, it is enough to show that for every $x\in (R:M), m\in M$, that $M\xrightarrow{\cdot xm}M$ factors through a finitely generated free $R$-module. And indeed, since $xM\subseteq R$, so the factoring is given by the commutative diagram 
\begin{tikzcd}
M \arrow[rd, "\cdot x"'] \arrow[rr, "\cdot xm"] &                    & M \\
                                    & R \arrow[ru, "\cdot m"'] &    
\end{tikzcd} 

To show the reverse inclusion $L\subseteq \tr_R(M)$, let $r\in L$ i.e. suppose $r\in R$ and we have a commutative diagram $$\begin{tikzcd}
M \arrow[rd, "g"'] \arrow[rr, "\cdot r"] &                    & M \\
                                    & R^{\oplus n} \arrow[ru, "f"'] &    
\end{tikzcd}$$ for some integer $n\ge 0$ and some $R$-linear maps $f,g$. Since $M$ contains a non-zero-divisor of $R$, so $rM\ne 0$, so $n> 0$. 
Let $\pi_i,j_i$ denote the $i$-th coordinate projection and inclusion maps from $R^{\oplus n} \to R$ and $R\to R^{\oplus n}$ respectively. Writing $g_i=\pi_i\circ g:M\to R$, and $f_i=f\circ j_i:R\to M$ respectively, we see that $g=(g_1,\cdots,g_n)$ and $f(r_1,\cdots,r_n)=\sum_{i=1}^n f_i(r_i)=\sum_{i=1}^nr_if_i(1), \forall (r_1,...,r_n)\in R^{\oplus n}$.  Since each $g_i\in \Hom_R(M,R)$, so by \cite[Proposition 2.4(1)]{trace}, there exists   $q_i\in (R:M)\subseteq Q(R)$ such that $g_i(x)=q_ix, \forall x\in M$.  Now let $b\in M \cap R$ be a non-zero-divisor. Then, due to the commutative diagram, we have 

$$rb=f(g(b))=f(g_1(b),\cdots,g_n(b))=f(q_1b,\cdots,q_nb)=\sum_{i=1}^n f_i(q_ib)=\sum_{i=1}^n q_ibf_i(1)=\left(\sum_{i=1}^n q_if_i(1)\right)b$$, where all these equalities takes place in $Q(R)$.  Since $b$ is a non-zero-divisor on $R$, so it is a non-zero-divisor on $Q(R)$, so $r=\sum_{i=1}^n q_if_i(1)\in (R:M)M$.

\end{proof}  

\section{some applications to annihilators of Ext and Tor of one-dimensional local Cohen--Macaulay rings} 

In this section, let $(R,\m)$ be a local Cohen--Macaulay ring of dimension $1$ with $\m$-adic completion $\widehat R$. Since $\c(R)=(R:\overline R)$ is an ideal of $R$ and $\overline R$ (see \ref{conductor}), so $\widehat{\c(R)}$ is an ideal of $\widehat R$ and $\widehat R\otimes_R \overline R=\overline{\widehat R}$ (see the discussion of \cite[Remark 4.8]{lw}). So, $\widehat{\c(R)}\subseteq \c(\widehat R)$ (see the discussion of \ref{conductor}).

As our first application of Theorem \ref{traceann}, we give an affirmative answer to \cite[Question 4.7]{dt}, relating to equality of annihilator of $\Tor$ and $\Ext$ of all modules in $\cm_0(R)$, when $R$ is a local ring of dimension $1$ whose completion is reduced (hence $R$ is Cohen--Macaulay). More precisely, we prove the following result:  

\begin{prop}\label{41} Let $(R,\m)$ be a local ring of dimension one whose completion is reduced, so that $\cm(R)=\cm_0(R)$ holds. Let $\X$ be a subcategory of $\Mod R$ containing $\syz^2\cm(R)$ and $\Y$ be a subcategory of $\Mod R$ containing $\cm(R)$. Then, it holds that $\e^0(\cm(R),\X)=\e^1(\cm(R),\X)=\t_0(\cm(R),\Y)=\c(R)$. In particular, it holds that $\e^0(\cm(R))=\e^1(\cm(R))=\t_0(\cm(R))=\c(R)$.  
\end{prop}  

\begin{proof} Since $\widehat R$ is reduced, so $R$ is reduced, so $R$ is Cohen--Macaulay with an isolated singularity (as $\dim R$=1). So, $\cm(R)=\cm_0(R)$. Clearly, $\e^0(\cm(R),\X)\subseteq \e^1(\cm(R),\X)$. Since $\widehat R$ is reduced, so $\overline R$ is a finite $R$-module, hence $\c(R)=(R:\overline R)$ contains a non-zero-divisor of $R$. Moreover, $\c(R)\cong \Hom_R(\overline R,R)$ is the $R$-dual of the finite $R$-module $\overline R$, so $\c(R)$ is a second syzygy module over $R$. So, $\c(R)\cong F\oplus  \syz^2_R M $, for some finite free $R$-module $F$, where $\syz^2_R M$ denotes the second-syzygy in some minimal free resolution for some $R$-module $M$. Writing $N=\syz_R M$, we see $N\in \cm(R)$ and $\c(R)\cong F\oplus \syz_R N\in \syz \cm(R)$. So, $\syz_R \c(R)\in \syz^2\cm(R)\subseteq \X$. So we have, $\e^1(\cm(R),\X)\subseteq \ann_R \Ext^2_R(N,\syz_R \c(R))$, and we also have the following inclusions and equalities $$ \ann_R \Ext^2_R(N,\syz_R \c(R))=\ann_R \Ext^1_R(F\oplus \syz_R N,\syz_R \c(R))=\ann_R\Ext^1_R(\c(R),\syz_R \c(R))\overset{(1)}=\tr_R(\c(R))\overset{(2)}=\c(R)$$ where equality (1) holds by Theorem \ref{traceann} since $\c(R)$ is an ideal of $R$ containing a non-zero-divisor (as the completion of $R$ is reduced, see the last sentence of the discussion in \ref{conductor}); and equality (2) holds by the discussion in \ref{trcon}.  
This shows $\e^0(\cm(R),\X)\subseteq \e^1(\cm(R),\X)\subseteq \c(R)$.   

Similarly, $\syz_R \Tr \c(R)\in \cm(R)\subseteq \Y$, so we get  
$$\t_0(\cm(R),\Y)\subseteq \ann_R \Tor^R_1( N,\syz_R \Tr \c(R))=\ann_R \Tor^R_2(N,\Tr \c(R))=\ann_R \Tor^R_1(F\oplus \syz_R N,\Tr \c(R))$$ 
and
$\ann_R \Tor^R_1(F\oplus \syz_R N,\Tr \c(R))=\ann_R \Tor^R_1(\c(R),\Tr \c(R))\overset{(3)}=\tr_R(\c(R))=\c(R)$, where (3) again follows by Theorem \ref{traceann}.  This shows $\t_0(\cm(R),\Y)\subseteq \c(R)$.   

So now it is enough to prove that $\c(R) \subseteq \e^0(\cm(R),\X) \cap \t_0(\cm(R),\Y)$. In fact, we will observe that $\c(R) \subseteq \e^0(\cm(R),\Mod R) \cap \t_0(\cm(R),\Mod R)$. By Lemma \ref{lemm}, it is enough to observe that $\c(R)\subseteq \e^0(M,\Mod R)$ for every $M\in \cm(R)$.  Since $\widehat R$ is reduced and one-dimensional, and if $M\in \cm(R)$, then $\widehat M\in \cm(\widehat R)$, so for every $i\ge 1$ and every $N\in \mod R$, we get by \cite[Proposition 3.1]{wa} that $0=\c(\widehat R)\Ext^i_{\widehat R}(\widehat M, \widehat N)=\c(\widehat R)\widehat{\Ext^i_R(M,N)}$. Since $\widehat{\c(R)}\subseteq \c(\widehat R)$ by the discussion preceding this result, so we get $\widehat{\c(R)}\widehat{\Ext^i_R(M,N)}=0$. Hence, $\widehat{\c(R)\Ext^i_R(M,N)}=0$, so $\c(R)\Ext^i_R(M,N)=0$. Since this is true for any $N\in \mod R$, so in particular, take the first syzygy $\syz_R M$ in a resolution of $M$ by finite free $R$-modules, so that $\syz_R M\in \mod R$. Hence, $\c(R)\Ext^1_R(M,\syz_R M)=0$, so $\c(R)\subseteq \ann_R \Ext^1_R(M,\syz_R M)=\e^0(M,\Mod R)=\t_0(M,\Mod R)$ by Lemma \ref{lemm}. 

To conclude, we have shown $\c(R)\subseteq \e^0(\cm(R),\Mod R)\subseteq \e^0(\cm(R),\X)\subseteq \e^1(\cm(R),\X)\subseteq \c(R)$, and  $\c(R)\subseteq \t_0(\cm(R),\Mod R)\subseteq \t_0(\cm(R),\Y)\subseteq \c(R)$. This gives all the equalities as claimed in the Proposition. The last part of the proposition follows by taking $\X=\Y=\cm(R)$.  
\end{proof}  

As our second application of Theorem \ref{traceann}, we give an alternative proof of one of the main Theorems of \cite{curve}, namely \cite[Theorem 4.3 and 5.10]{curve}

\begin{prop}\label{cacon} Let $(R,\m)$ be a local Gorenstein ring of dimension one whose completion is reduced. Then, $\ca(R)=\ca^n(R)=\c(R)$ for every $n\ge 2$.   
\end{prop}  

\begin{proof}  Pick $s>1$ large enough so that $\ca(R)=\ca^{s+1}(R)$. Since $R$ is Gorenstein and $\c(R)$ is maximal Cohen--Macaulay (as it is an ideal), so by \cite[Construction 12.10]{lw} we observe that $\c(R)\in \syz^s_R \mod R$, hence $\c(R)\cong F \oplus \syz^s_R M$ for some $M\in \mod R$ and free $R$-module $F$. So then, $\ca(R)=\ca^{s+1}(R)\subseteq \ann_R \Ext^{s+1}_R(M,\syz_R \c(R))=\ann_R \Ext^1_R(F\oplus \syz^s_R M, \syz_R \c(R))=\ann_R\Ext^1_R(\c(R),\syz_R \c(R))=\tr_R(\c(R))=\c(R)$, where the last two equalities holds by Theorem \ref{traceann} and discussion \ref{conductor}, \ref{trcon}. So this shows $\ca(R)\subseteq \c(R)$. Now, we also have $\syz \mod R \subseteq \cm(R)$, so $\e^0(\cm(R),\mod R)\subseteq \e^1(\mod R,\mod R)=\ca^2(R)\subseteq \ca^n(R) \subseteq \ca(R)$ for every $n\ge 2$. Hence by Proposition \ref{41} we get $\c(R)\subseteq \e^0(\cm(R),\mod R) \subseteq \ca^2(R)\subseteq \ca^n(R) \subseteq \ca(R)$ for every $n\ge 2$. This combined with $\ca(R) \subseteq \c(R)$ gives the required claim. 
\end{proof}

\end{document}